\documentclass[]{amsart}

\usepackage{xcite}

\usepackage{amsthm}
\newcommand{\old}[1]{}

\renewcommand{\emph}[1]{\textit{#1}}
\usepackage{enumerate,amsmath,latexsym,amssymb}
\usepackage{bm}
\usepackage{chapterbib} 

\usepackage{pdftricks}
\begin{psinputs}
  \usepackage{pstricks}
\end{psinputs}
\usepackage{color}\usepackage{graphicx}

\definecolor{brown}{cmyk}{0, 0.72, 1, 0.45}
\definecolor{grey}{gray}{0.5}

 \allowdisplaybreaks

\newcounter{rot}

\newcommand{\ignore}[1]{}

\newcommand{\set}[1]{\left\{#1\right\}}

\def\ii_(#1,#2){i_{#1}^{#2}}

\def\cM{\mathcal{M}}
\def\a{\alpha}

\def\d{\delta}

\def\e{\varepsilon}

\def\p{\pi}

\def\r{\rho}

\def\s{\sigma}

\def\om{\omega}

\def\cT{{\mathcal T}}

\def\Re{\mathbb{R}}

\newcommand{\brac}[1]{\left( #1 \right)}

\newcommand{\expect}{\operatorname{\bf E}}
\def\E{\expect}
\renewcommand{\Pr}{\operatorname{\bf Pr}}
\newcommand\bfrac[2]{\left(\frac{#1}{#2}\right)}

\newtheorem{theorem}{Theorem}[section]

\newtheorem{observation}[theorem]{Observation}
\newcounter{thmtemp}
{
\theoremstyle{definition}
\newtheorem{definition}[theorem]{Definition}

\newtheorem{remark}[theorem]{Remark}

}

\newcommand{\nospace}[1]{}

\def\path{\operatorname{PATH}}

\renewcommand{\Re}{\mathbb{R}}

\makeatletter
\newcommand{\arabicifpos}[1]{\@arabicifpos{\@nameuse{c@#1}}}
\newcommand{\@arabicifpos}[1]{\ifnum #1>0 \number #1\fi}
\makeatother

\newcommand{\ep}{\varepsilon}

\title{Separating effect from significance in Markov chain tests}

\author{Maria Chikina}
\address{Department of Computational and Systems Biology\\
University of Pittsburgh\\
3078 Biomedical Science Tower 3\\
Pittsburgh, PA 15213\\
U.S.A.}
\email[email: ]{mchikina@pitt.edu}
\author{Alan Frieze}
\email[email: ]{alan@random.math.cmu.edu}
\thanks{Research supported in part by NSF grant DMS-1362785.}
\author{Jonathan Mattingly}
\email[email: ]{jonm@math.duke.edu}
\author{Wesley Pegden}
\email[email: ]{wes@math.cmu.edu}
\thanks{Research supported in part by NSF grant DMS-1363136 and the Sloan foundation.}
\address{Department of Mathematical Sciences\\
Carnegie Mellon University\\
Pittsburgh, PA 15213\\
U.S.A.}

\date{\today}
\begin{document}
\begin{abstract}
We give qualitative and quantitative improvements to theorems which enable significance testing in Markov Chains, with a particular eye toward the goal of enabling strong, interpretable, and statistically rigorous claims of political gerrymandering.  Our results can be used to demonstrate at a desired significance level that a given Markov Chain state (e.g., a districting) is extremely unusual (rather than just atypical) with respect to the fragility of its characteristics in the chain.  We also provide theorems specialized to leverage quantitative improvements when there is a product structure in the underlying probability space, as can occur due to geographical constraints on districtings.
\end{abstract}

\maketitle
\section{Motivation}
At its core, this note discusses improvements on a number of theorems
for significance testing in Markov Chains. The improvements to the Theorem statements are both qualitative and quantitative to enable strong, easily interpretable statistical claims, and include extensions to settings where more structural assumptions lead to huge
improvements in the bounds. This class of  theorems is
particular interest because they do not assume that the chain has
converged to equilibrium. This can be of huge practical importance.

Yet, this tells only part of the story. The development of this class
of algorithms and these particular extensions have been directly
motivated by a question of great contemporary interest; detecting and
quantifying gerrymandering.

The definiteness and correctness provided by these theorem provide
substantial weight in a legal setting. The basic recipe in the
gerrymandering context is the following. One starts a reversible
Markov change from a particular redistricting map which claims to be
typical among maps one which the Markov
chain's invariant distribution  is concentrated.

Operationally, this allows one to rigorously assess the likelihood
of choosing a particular map if one was only considered a specific
collection of non-partisan considerations. These methods (and theorems)
have been used successfully by one of the  authors in Gerrymandering
court cases in Pennsylvania and North Carolina.

The first part of this article gives new results along these lines, extending the work in \cite{outliers} to allow separation of effect size from the quantification of statistical significance. The second
part, in Section \ref{s.product}, develops versions of some of these results in a special setting with a particular structure on the probability space motivated by recent legal proceedings.  In particular, in balancing the federal
one-person-one-vote mandate with the ``keep counties whole'' prevision
of the North Carolina Constitution, the North Carolina courts ruled in
 \emph{Stephenson v.~Bartlett} that a particular algorithm should be used to ``cluster'' the counties into independent county groups which are districted separately.     This gives a product structure to the underlying probability space which can be exploited in theorems designed to take advantage of it.

\section{Introduction}

Consider a reversible Markov Chain $\cM$ whose state-space $\Sigma$ is endowed with some labeling $\omega:\Sigma\to \Re$, and for which $\pi$ is a stationary distribution.   $\cM$, $\pi$, $\omega$, and a fixed integer $k$ determine a vector
\[
p^k_0,p^k_1,\dots,p^k_k
\]
where for each $i$, $p^k_i$ is the probability that for a $k$-step $\pi$-stationary trajectory $X_0,\dots,X_k$, the minimum $\omega$ value occurs at $X_i$.  In other words, $p^k_i$ is the probability that if we choose $X_0$ randomly from the stationary distribution $\pi$ and take $k$ steps in $\cM$ to obtain the the trajectory $X_0,X_1,\dots,X_k$, that we observe that $\omega(X_i)$ is the minimum among $\omega(X_0),\dots,\omega(X_k)$.  Note that if we adopted the convention that we break ties among the values $\omega(X_0),\dots,\omega(X_k)$ randomly, we would have that $p^k_0+\cdots+p^k_k=1$, for any $\cM, \pi,$ and $k$.

At first glance, it might be natural to assume that we must have something like $p^k_i\approx \frac{1}{k+1}$ for all $0\leq i\leq k$.  But this is actually quite far from the truth; \cite{outliers} showed that for some $\cM,\pi,k$, we can have $p^k_0$ as large as essentially $\frac{1}{\sqrt{2\pi k}}$.

As shown in \cite{outliers}, this is essentially the worst possible behavior for $p^k_0$.  In particular, we can generalize the vector $\{p^k_i\}$ defined above as possible: let us define, given $\cM, \pi$, $k$, and $\ep$, the vector
\[
p^k_{0,\ep},p^k_{1,\ep},\dots,p^k_{k,\ep}
\]
where each $p^k_{i,\ep}$ is the probability that $\omega(X_i)$ is among the smallest $\ep$ values in the list $\omega(X_0),\dots,\omega(X_k)$.  Then in \cite{outliers} we proved:

\begin{theorem}\label{t.rootep}
  Given a reversible Markov chain $\cM$ with stationary distribution $\pi$, an $\ep>0$, $k\geq 0$, and With $p^k_{i,\ep}$ defined as above, we have that
  \[
  p^k_{0,\ep}\leq \sqrt{2\e}.
  \]
\end{theorem}
Note that the example from \cite{outliers} realizing $p^k_0\approx \frac{1}{\sqrt{2\pi k}}$ shows that this theorem is best possible, up to constant factors.


One important application of Theorem \ref{t.rootep} is that it characterizes the statistical significance associated to the result of a natural test for gerrymandering of political districtings.  In particular, consider the following general procedure to evaluate a districting of a state:\\
\smallskip

\textbf{Local Outlier Test}
\begin{enumerate}
\item \label{begin} Beginning from the districting being evaluated,
\item \label{sequence} Make a sequence of random changes to the districting, while preserving some set of constraints imposed on the districtings.
\item Evaluate the partisan properties of each districting encountered (e.g., by simulating elections using past voting data).
\item \label{crafted} Call the original districting ``carefully crafted'' or ``gerrymandered'' if the overwhelming majority of districtings produced by making small random changes are less partisan than the original districting.
\end{enumerate}

Naturally, the test described above can be implemented so that it precisely satisfies the hypotheses of Theorem \ref{t.rootep}.  For this purpose, a (very large) set of comparison districtings are defined, to which the districting being evaluated belongs.  For example, the comparison districtings may be the districtings built out of Census blocks (or some other unit) which are contiguous, equal in population up to some specified deviation, or include other constraints.   A Markov chain $\cM$ is defined on this set of districtings, where transitions in the chain correspond to changes in districtings.  (For example, a transition may correspond to randomly changing the district assignment of a randomly chosen Census block which currently borders more than one district, subject to the constraints imposed on the comparison set.)  The ``random changes'' from Step \ref{sequence} will then be precisely governed by the transition probabilities of the Markov chain $\cM$.  By designing $\cM$ so that the uniform distribution $\pi$ on the set of comparison districtings $\Sigma$ is a stationary distribution for $\cM$, Theorem \ref{t.rootep} gives an upper bound on the false-positive rate (in other words, global statistical significance) for the ``gerrymandered'' declaration when it is made in Step \ref{crafted}.

Apart from its application to gerrymandering, Theorem \ref{t.rootep} has a simple informal interpretation for the general behavior of reversible Markov chains, namely: \emph{typical (i.e., stationary) states are unlikely to change in a consistent way under a sequence of chain transitions}, with a best-possible quantification of this fact (up to constant factors).

Also, in the general setting of a reversible Markov chain, the theorem leads to a simple quantitative procedure for asserting rigorously that $\sigma_0$ is atypical with respect to $\pi$ without knowing the mixing time of $\cM$: simply observe a random trajectory $\sigma_0=X_0,X_1,X_2\dots,X_k$ from $\sigma_0$ for any fixed $k$.  If $\omega(\sigma_0)$ is an $\ep$-outlier among $\omega(X_0),\dots,\omega(X_k)$, then this is statistically significant at $\sqrt{2\ep}$ against the null hypothesis that $\sigma_0\sim \pi$.

This quantitative test is potentially useful because $\sqrt{2\ep}$ converges quickly enough to 0 as $\ep\to 0$; in particular, it is possible to obtain good statistical significance from observations which can be made with reasonable computational resources.  Of course, faster convergence to 0 would be even better, but, as already noted, $p\approx \sqrt{\ep}$ is roughly a best possible upper bound.

Unknown to the authors at the time of the publication of \cite{outliers}, a 1989 paper of Besag and Clifford described a test related to that based on Theorem \ref{t.rootep}, which has essentially a one-line proof, which we discuss in Section \ref{s.back}:
\begin{theorem}[Besag and Clifford serial test]\label{t.GCserial}
  Fix any number $k$ and suppose that $\sigma_0$ is chosen from a stationary distribution $\pi$, and that $\xi$ is chosen uniformly in $\{0,\dots,k\}$.  Consider two independent trajectories $Y_0,Y_1,\dots$ and $Z_0,Z_1,\dots$ in the reversible Markov Chain $\cM$ (whose states have real-valued labels) from $Y_0=Z_0=\sigma_0$. If we choose $\sigma_0$ from a stationary distribution $\pi$ of $\cM$, then for any $k$ we have that
  \[
\Pr\brac{\om(\s_0)\text{\textnormal{ is an $\ep$-outlier among }}\om(\s_0),\om(Y_1),\dots,\om(Y_\xi),\om(Z_1),\dots,\om(Z_{k-\xi})}\leq\ep.
  \]
\end{theorem}
Here, a real number $a_0$ is an \emph{$\ep$-outlier} among $a_0,\dots,a_k$ if
\[
\#\set{i\in \{0,\dots,k\}\mid a_i\leq a_0}\leq \ep(k+1).
\]
In particular, the striking thing about Theorem \ref{t.GCserial} is that it achieves a best-possible dependence on the parameter $\ep$.  (Notice that $\ep$ would be the correct value of the probability if, for example, the Markov chain is simply a collection of independent random samples.)  The sacrifice is in Theorem \ref{t.GCserial}'s slightly more complicated intuitive interpretation, which would be:  \emph{typical (i.e., stationary) states are unlikely to change in a consistent way under two sequences of chain transitions of random complementary lengths.}  In particular, in applications of these statistical tests to aspects of public policy, it is desirable to have tests with simple, intuitive interpretations.  To enable better significance testing in this sphere, one goal of the present note is to prove a theorem enabling Markov chain significance testing which is intuitively interpretable in the sense of Theorem \ref{t.rootep}, while having linear dependence on $\ep$, as in Theorem \ref{t.GCserial}.

\bigskip

One common feature of the tests based on Theorem \ref{t.rootep} and \ref{t.GCserial} is the use of randomness.  In particular, the probability space at play in these theorems includes both the random choice of $\sigma_0$ assumed by the null hypothesis and the random steps taken by the Markov chain from $\sigma_0$.  Thus the measures of ``how (globally) unusual'' $\sigma_0$ is with respect to its performance in the local outlier test and ``how sure'' we are that $\sigma_0$ is unusual in this respect are intertwined in the final $p$-value.  In particular, the effect size and the statistical significance are not explicitly separated.

To further the goal of simplifying the interpretation of the results of these tests, our approach in this note will also show that tests like these can be efficiently used in a way which separates the measure of statistical significance from the question of the magnitude of the effect.  In particular, recalling the probabilities $p_{0,\ep}^k,\dots,p_{k,\ep}^k$ defined previously, let us define the probability $p_{0,\ep}^k(\sigma_0)$ to be the probability that on a trajectory $\sigma_0=X_0,X_1,\dots,X_k$, $\omega(\sigma_0)$ is among the smallest $\ep$ fraction of the list $\omega(X_0),\dots,\om(X_k)$.  Now we make the following definition:
\begin{definition}\label{d.outlier}
With respect to $k$, the state $\sigma_0$ is an \emph{$(\ep,\alpha)$-outlier} in $\cM$ if, among all states in $\cM$, $p_{0,\ep}^k(\sigma_0)$ is in the largest $\alpha$ fraction of the values of $p_{0,\ep}^k(\sigma)$ over \emph{all} states $\sigma\in \cM$, weighted according to $\pi$. 
\end{definition}
In particular, being an $(\ep,\alpha)$-outlier measures the likelihood of $\sigma_0$ to fail the local outlier test, ranked against \emph{all other states} $\sigma\sim \pi$ of the chain $\cM$.   For example, fix $k=10^9$.  If $\sigma_0$ is a $(10^{-6},10^{-5})$-outlier in $\cM$ and $\pi$ is the uniform distribution, this means that among \emph{all} states $\s\in \cM$, $\sigma_0$ is more likely than all but a $10^{-5}$ fraction of states to have an $\omega$-value in the bottom $10^{-6}$ values $\omega(X_0),\omega(X_1),\dots,\omega(X_{10^9})$.  Note that the probability space underlying the ``more likely'' claim here just concerns the choice of the random trajectory $X_1,\dots,X_{10^9}$ from $\cM$.

Note that whether $\sigma_0$ is a $(\ep,\alpha)$-outlier is a deterministic question about the properties of $\sigma_0, \cM,$ and $\omega$.  Thus it is a deterministic measure (defined in terms of certain probabilities) of the extent to which $\sigma_0$ is unusual (globally, in all of $\cM$) with respect to it's local fragility in the chain.

The following theorem enables one to assert statistical significance for the property of being an $(\ep,\alpha)$-outlier.  In particular, while tests based on Theorems \ref{t.rootep} and \ref{t.GCserial} take as their null hypothesis that $\sigma_0\sim \pi$, the following theorem takes as its null hypothesis merely that $\sigma_0$ is not an $(\ep,\alpha)$-outlier.  

\begin{theorem}\label{t.outlier}
  Consider $m$ independent trajectories
  \begin{align*}
    \cT^1= &\, (X^1_0,X^1_1,\dots,X^1_k),\\
    &\vdots\\
    \cT^m= &\, (X^m_0,X^m_1,\dots,X^m_k)
  \end{align*}
  of length $k$ in the reversible Markov Chain $\cM$ (whose states have real-valued labels) from a common starting point $X^1_0=\dots=X^m_0=\sigma_0$.  Define the random variable $\rho$ to be the number of trajectories $\cT^i$ on which $\sigma_0$ is an $\ep$-outlier.

  If $\sigma_0$ is not an $(\ep,\alpha)$-outlier, then 
  \begin{equation}\label{l.rhobound}
  \Pr\brac{\rho\geq m\sqrt{\tfrac{2\ep}{\alpha}}+r}\leq e^{-\min(r^2\sqrt{\alpha/2\ep}/3m,r/3)}.
  \end{equation}
\end{theorem}
In particular, apart from separating measures of statistical significance from the quantification of a local outlier, Theorem \ref{t.outlier} connects the intuitive Local Outlier Test tied to Theorem \ref{t.rootep} (which motivates the definition of a $(\ep,\alpha)$-outlier) to the better quantitative dependence on $\ep$ in Theorem \ref{t.GCserial}.

To compare the quantitative performance of Theorem \ref{t.outlier} to Theorems \ref{t.rootep} and \ref{t.GCserial}, consider the case of a state $\sigma_0$ for which a random trajectory $\sigma_0=X_0,X_1,\dots,X_k$ is likely (say with some constant probability $p'$) to find $\sigma_0$ an $\ep'$-outlier.  For Theorem \ref{t.rootep}, significance at $p\approx \sqrt{2\ep}$ would be obtained\footnote{Multiple tests have limited utility here or with Theorem \ref{t.GCserial} since there is no independence (the null hypothesis $\sigma_0\sim \pi$ is not being resampled).  In particular, multiple runs might be done merely until a trajectory is seen on which $\sigma_0$ is indeed an $\ep'$ outlier (requiring $1/p'$ runs, on average), in conjunction with multiple hypothesis testing.}, while using Theorem \ref{t.GCserial}, one would hope to obtain significance of $\approx \ep'$.    Applying Theorem \ref{t.outlier}, we would expect to see $\rho$ around $m\cdot p'$.  In particular, we could demonstrate that $\sigma_0$ is an $(\ep',\alpha)$ outlier for $\alpha=\frac{3\ep}{(p')^2}$ (a linear dependence on $\ep$) at a $p$-value which can be made arbitrarily small (at an exponential rate) as we increase the number of observed trajectories $m$.  As we will see in Section \ref{s.outlier}, the exponential tail in \eqref{l.rhobound} can be replaced by a binomial tail.  In particular, the following special case applies:
\begin{theorem}\label{t.outall}
  With $\cT^1,\dots,\cT^m$ as in Theorem \ref{t.outlier}, we have that if $\s_0$ is not an $(\ep,\alpha)$ outlier, then
  \[
  \Pr\brac{\sigma_0\text{\textnormal{ an $\ep$-outlier on all of }}\cT^1,\dots,\cT^m}\leq \bfrac{2\ep}{\alpha}^{m/2}.
  \]
\end{theorem}
Theorem \ref{t.outall} also has advantages from the standpoint of avoiding the need to correct for multiple hypothesis testing, as we discuss in Section \ref{s.multiple}.

\bigskip

To prove Theorem \ref{t.outlier}, we will prove the following, which has a quantitative dependence on $\ep$ which is nearly as strong as in Theorem \ref{t.GCserial}, while eliminating the need for the random choice of $\xi$ there.

\begin{theorem}\label{t.twopaths}
  Consider two independent trajectories $Y_0,\dots,Y_k$ and $Z_0,\dots,Z_k$ in the reversible Markov Chain $\cM$ (whose states have real-valued labels) from a common starting point $Y_0=Z_0=\sigma_0$. If we choose $\sigma_0$ from a stationary distribution $\pi$ of $\cM$, then for any $k$ we have that
  \[
\Pr\brac{\om(\s_0)\text{\textnormal{ is an $\ep$-outlier among }}\om(\s_0),\om(Y_1),\dots,\om(Y_k),\om(Z_1),\dots,\om(Z_k)}< 2\ep.
  \]
\end{theorem}
Note that Theorem \ref{t.twopaths} is equivalent to the statement that the probabilities $p^k_{i,\ep}$ always satisfy
\begin{equation}\label{l.restatetwopaths}
p^{2k}_{k,\ep}<2\ep.
\end{equation}
\begin{remark}\label{r.badcase}
  
As in the case of Theorem \ref{t.rootep}, it seems like an interesting question to investigate the tightness of the constant 2; we will see in Section \ref{s.product} that there are settings where the impact of this constant is inflated to have outsize-importance. We point out here that at least for the case of $k=1$, $\ep=1/3$, $\rho^{2}_{1,\frac 1 3}$ can be at least as large as $\frac 1 2$, showing that the constant 2 in \eqref{l.restatetwopaths} cannot be replaced by a constant less than $\tfrac 3 2$, in general.  To see this, consider, for example, a bipartite complete graph $K_{n,n}$, where the labels of the vertices of one side are $1,\dots,n$ and the other are $n+1,\dots,2n$.  For the Markov chain given by the random walk on this undirected graph, we have that $\rho^{2}_{1,\frac 1 3}=\frac 1 2$.  Note that for this example, it is still the case that $\rho^{2k}_{k,\ep}\to \ep$ as $k\to \infty$, leaving open the possibility that the $2$ in \eqref{l.restatetwopaths} can be replaced with an expression asymptotically equivalent to 1.
\end{remark}


The following theorem is the analog of Theorem \ref{t.outlier} obtained when one uses an analog of Besag and Clifford's Theorem \ref{t.GCserial} in place of \ref{t.twopaths} in the proof.  This version pays the price of using a random $k$ instead of a fixed $k$ for the notion of an $(\ep,\alpha)$-outlier, but has the advantage that the constant $2$ is eliminated from the bound.  (Note that as in Theorem \ref{t.outlier}, the notion of $(\ep,\alpha)$-outlier used here is still just defined with respect to a single path, although Theorem \ref{t.GCserial} depends on using two independent trajectories.)  

\begin{theorem}\label{t.GCoutlier}
  Consider $m$ independent trajectories
  \begin{align*}
    \cT^1= &\, (X^1_0,X^1_1,\dots,X^1_{k_1}),\\
    &\vdots\\
    \cT^m= &\, (X^m_0,X^m_1,\dots,X^m_{k_m})
  \end{align*}
  in the reversible Markov Chain $\cM$ (whose states have real-valued labels) from a common starting point $X^1_0=\dots=X^m_0=\sigma_0$, where each of the lengths $k_i$ are independently drawn random numbers from a geometric distribution.  Define the random variable $\rho$ to be the number of trajectories $\cT^i$ on which $\sigma_0$ is an $\ep$-outlier.

  If $\sigma_0$ is not an $(\ep,\alpha)$-outlier with respect to $k$ drawn from the geometric distribution, then
    \begin{equation}\label{l.GCrhobound}
  \Pr\brac{\rho\geq m\sqrt{\tfrac{\ep}{\alpha}}+r}\leq e^{-\min(r^2\sqrt{\alpha/\ep}/3m,r/3)}.
  \end{equation}
\end{theorem}
Again, there is an analogous version to Theorem \ref{t.outall}, where $2\ep$ is replaced by $\ep$.

\bigskip

In their paper, Besag and Clifford also describe a parallel test, which we will discuss in Section \ref{s.starsplit}.   In particular, in Section \ref{s.starsplit} we will describe a test which generalizes Besag and Clifford's serial and parallel tests in a way which could be useful in certain parallel regimes.

Finally, we consider an interesting case in the analysis of districtings that arises when the districting problem can be decomposed into several non-interacting districting problems; for example, for the districting for the state Senate of North Carolina, the state is divided into 29 ``county clusters'', each corresponding to a prescribed number of districts based on their populations, so that a districting of the whole state is obtained by non-interacting districting processes in these different county clusters.  In this case, the probability space of random districtings is really a product space, and this structure can be exploited in a strong way for the statistical tests developed in this manuscript.  We develop results for this setting in Section \ref{s.product}.

\section{Multiple hypothesis considerations}
\label{s.multiple}
When applying Theorem \ref{t.outlier} directly, one cannot simply run $m$ trajectories, observe the list $\ep_1,\ep_2,\dots,\ep_m$ where each $\ep_i$ is the minimum $\ep_i$ for which $\sigma_0$ is an $\ep_i$-outlier on $\cT^i$, and then, post-hoc, freely choose the parameters $\alpha$ and $\ep$ in Theorem \ref{t.outlier} to achieve some desired trade-off between $\alpha$ and the significance $p$.

The problem, of course, is that in this case one is testing multiple hypotheses (infinitely many in fact; one for each possible pair $\ep$ and $\alpha$) which would require a multiple hypothesis correction.

One way to avoid this problem is to essentially do a form of cross validation, were a few trajectories are run for the purposes of selecting suitable $\ep$ and $\alpha$, and then discarded from the set of trajectories from which we obtain significance.

A simpler approach, however, is to simply set the parameter $\ep=\ep_{(t)}$ as the $t$th-smallest element of the list $\ep_1,\dots,\ep_m$ for some fixed value $t$.  The case $t=m$, for example, corresponds to taking $\ep$ as the maximum value, leading to the application of Theorem \ref{t.outall}.

The reasons this avoids the need for a multiple hypothesis correction is that we can order our hypothesis events by containment.  In particular, when we apply this test with some value of $t$, we will always have $\rho=t$.  Thus the significance obtained will depend just on the parameter $\ep_{(t)}$ returned by taking the $t$-th smallest $\ep_i$ and on our choice of $\alpha$ (as opposed to say, the particular values of the other $\ep_i$'s which are not the $t$-th smallest).  In particular, regardless of how we wish to trade-off the values of $\alpha$ and $p$ we can assert from our test, our optimum choice of $\alpha$ (for our fixed choice of $t$) will depend just on the value $\ep_{(t)}$.  In particular, we can view $\alpha$ as a function $\alpha(\ep_{(t)})$, so that we when applying Theorem \ref{t.outlier} with with $\ep=\ep_{(t)}$, we are evaluating the single-parameter infinite family of hypotheses $H_{\ep_{(t)},\alpha(\ep_{(t)})}$, and we do not require multiple hypothesis correction since the hypotheses are nested; i.e., since
\begin{equation}\label{l.nested}
\ep_{(t)}\leq \ep'_{(t)}\implies H_{\ep_{(t)},\alpha(\ep_{(t)})} \subseteq H_{\ep'_{(t)},\alpha(\ep_{(t)})}.
\end{equation}
Indeed, \eqref{l.nested} implies that
\[
\Pr\brac{\bigcup_{\ep_{(t)}\leq \beta} H_{\ep_{(t)},\alpha(\ep_{(t)})}}=\Pr(H_{\beta,\alpha(\beta)}),
\]
which ensures that when applying Theorem \ref{t.outlier} in this scenario, the probability of returning a $p$-value $\leq p_0$ for any fixed value $p_0$ will indeed be at most $p_0$.

\section{Proof background}
\label{s.back}

We begin this section by giving the proof of Theorem \ref{t.twopaths}.  In doing so we will introduce some notation that will be useful throughout the rest of this note.  To make things as accessible as possible, we give every detail of the proof.

In this manuscript, a Markov Chain $\cM$ on $\Sigma$ is specified by the transition probabilities $\{\pi_{\s_1,\s_2}\mid \s_1,\s_2\in \Sigma\}$ of a chain.  A \emph{trajectory} of $\cM$ is a sequence of random variables $X_0,X_1,\dots$ required to have the property that for each $i$ and $\s_0,\dots,\s_i$, we have
\begin{equation}\label{pij}
  \Pr\brac{X_i=\s_i\mid X_{i-1}=\s_{i-1},X_{i-2}=\s_{i-2}\dots,X_{0}=\s_0}\\
 =\pi_{\s_i,\s_{i-1}}.
\end{equation}
In particular, the Markov property of the trajectory is that the conditioning on $X_{i-2},X_{i-3},\dots$ is irrelevant once we condition on the value of $X_{i-1}$.  Recall that $\pi$ is a stationary distribution if $X_0\sim \pi$ implies that $X_1\sim \pi$ and thus also that $X_i\sim \pi$ for all $i\geq 0$; in this case we that the trajectory $X_0,X_1,\dots$ is \emph{$\p$-stationary}.  The Markov Chain $\cM$ is \emph{reversible} if any $\pi$-stationary trajectory $X_0,\dots,X_k$ is equivalent in distribution to its reverse $X_k,\dots,X_0$.

We say that $a_j$ is \emph{$\ell$-small} among $a_0,\dots,a_s$ if there are at most $\ell$ indices $i\neq j$ among $0,\dots,s$ such that $a_i\leq a_j$.  The following simple definition is at the heart of the proofs of Theorems \ref{t.rootep}, \ref{t.twopaths}, \ref{t.GCserial}.
\begin{definition}\label{d.rhos}
  Given a Markov Chain $\cM$ with labels $\omega:\Sigma\to \Re$ and stationary distribution $\pi$, we define for each $\ell,j\leq k$ a real number $\r_{j,\ell}^k$, which is the probability that for a $\pi$-stationary trajectory $X_0,X_1,\dots,X_k$, we have that $\omega(X_j)$ is $\ell$-small among $\om(X_0),\dots,\om(X_k)$.
\end{definition}
Observe that \eqref{pij} implies that all $\pi$-stationary trajectories of a fixed length are all identical in distribution, and in particular, that the $\r_{j,\ell}^k$'s are well-defined.

Next observe that if the sequence of random variables $X_0,X_1,\dots$ is a $\pi$-stationary trajectory for $\cM$, then so is any interval of it.  For example,
\[
(X_{k-j},\dots,X_k,\dots,X_{2k-j})
\]
is another stationary trajectory, and thus the probability that $\omega(X_k)$ is $\ell$-small among $\omega(X_{k-j}),\dots,\omega(X_{2k-j})$ is equal to $\rho_{j,\ell}^k$.
In particular, since
\[
(\om(X_{k})\text{ is $\ell$-small among }\om(X_{k-j}),\dots,\om(X_{2k-j}))
\]
follows from
\[
(\om(X_{k})\text{ is $\ell$-small among }\om(X_0),\dots,\om(X_{2k}))
\]
for all $j=0,\dots,k$, we have that
\begin{equation}\label{l.compare}
 \r_{k,\ell}^{2k}\leq \r_{j,\ell}^k.
\end{equation}
We also have that $\sum\limits_{j=0}^k \r_{j,\ell}^k\leq \ell+1$.  Indeed, by linearity of expectation, this sum is the expected number of indices $j\in 0,\dots,k$ such that $\om(X_j)$ is $\ell$-small among $\om(X_0),\dots,\om(X_k)$.  Thus, averaging the left and right sides of \eqref{l.compare} over $j$ from $0$ to $k$, we obtain
\begin{equation}\label{l.implies}
\r_{k,\ell}^{2k}\leq \frac{\ell+1}{k+1}<2\cdot \frac{\ell+1}{2k+1}.
\end{equation}
Line \eqref{l.implies} already gives the theorem, once we make the following trivial observation:
\begin{observation}\label{o.stat}
  Under the hypotheses of Theorem \ref{t.twopaths}, we have that
  \[
  Y_k,Y_{k-1},\dots,Y_1,\s_0,Z_1,Z_2,\dots,Z_k
  \]
  is a $\pi$-stationary trajectory.
\end{observation}
This is an elementary consequence of the definitions, but since we will generalize this statement in Section \ref{s.starsplit}, we give all the details here:
\begin{proof}[Proof of Observation \ref{o.stat}]
Our hypothesis is that $Y_1,Y_2,\dots,Y_k$ and $Z_1,Z_2,\dots,Z_k$ are independent trajectories from a common state $Y_0=Z_0=\sigma_0$ chosen from the stationary distribution $\pi$.  Stationarity implies that
\[
(Z_0,Z_1,\dots,Z_k)\sim(X_k,X_{k+1},\dots,X_{2k}).
\]
Similarly, stationarity and reversibility imply that
\[
(Y_k,Y_{k-1},\dots,Y_0)\sim(X_0,X_1,\dots,X_k).
\]
Finally, our assumption that $Y_1,Y_2,\dots$ and $Z_1,Z_2,\dots$ are independent trajectories from $\sigma_0$ is equivalent to the condition that, for any $s_0,y_1,z_1,y_2,z_2,\dots, y_k,z_k\in \Sigma$, we have for all $j\geq 0$ that
\begin{multline}
  \Pr\brac{Z_j=z_j\mid Z_{j-1}=z_{j-1},\dots,Z_1=z_1,Z_0=Y_0=s_0,Y_1=y_1,\dots,Y_k=y_k}\\=
  \Pr\brac{Z_j=z_j\mid Z_{j-1}=z_{j-1},\dots,Z_1=z_1,Z_0=s_0}  
\end{multline}
Of course, since $\cM$ is a Markov Chain, this second probability is simply
\[
\Pr(Z_j=z_j\mid Z_{j-1}=z_{j-1})=\Pr(X_{k+j}=z_j\mid X_{k+j-1}=z_{j-1}).
\]
In particular, by induction on $j\geq 1$,
\[
(Y_k,Y_{k-1},\dots,Y_0=Z_0,Z_1,\dots,Z_j)\sim(X_0,X_{1},\dots,X_k,X_{k+1},\dots,X_{k+j}),
\]
and in particular
\begin{equation}\label{l.paste}
  (Y_k,\dots,\sigma_0,\dots,Z_k)\sim (X_0,\dots,X_k,\dots,X_{2k}).
\end{equation}
\end{proof}

Pared down to its bare minimum, this proof of Theorem \ref{t.twopaths} works by using that
%
$\r^{2k}_{k,\ell}$ is a lower bound on each $\r^k_{j,\ell}$, and then applying the simple inequality
\begin{equation}\label{e.sum}
\sum_{j=0}^k {\r^k_{j,\ell}}\leq \ell+1.
\end{equation}

The proof of Theorem \ref{t.GCserial} of Besag and Clifford is in some sense even simpler, using only \eqref{e.sum}, despite the fact that Theorem \ref{t.GCserial} has better dependence on $\ep$ (on the other hand, it is not directly applicable to $(\ep,\a)$-outliers in the way that we will use Theorem \ref{t.twopaths}).  Recall from Definition \ref{d.rhos} that the $\rho_{j,\ell}^k$'s are fixed real numbers associated to a stationary Markov Chain.  If $\ell,k$ are fixed and $\xi$ is chosen randomly from $0$ to $k$, then the resulting $\rho_{\xi,\ell}^k$ is a random variable uniformly distributed on the set of real numbers $\{\rho_{0,\ell}^k,\rho_{1,\ell}^k,\dots,\rho_{k,\ell}^k\}.$  In particular, Theorem \ref{t.GCserial} is proved by writing that the probability that $\om(\s_0)$ is $\ell$-small among $\om(\s_0),\om(Y_1),\dots,\om(Y_\xi),\om(Z_1),\dots,\om(Z_{k-\xi})$ is given by
\[
\frac 1{k+1}\brac{\rho_{0,\ell}^k+\rho_{1,\ell}^k+\dots+\rho_{k,\ell}^k}\leq \frac{\ell+1}{k+1},
\]
where the inequality is from \eqref{e.sum}.  Note that we are using an analog of Observation \ref{o.stat} to know that for any $j$, $Y_j,\dots,Y_1,\s_0,Z_1,Z_{k-j}$ is a $\pi$-stationary trajectory.

\section{Global significance for local outliers}
\label{s.outlier}

We now prove Theorem \ref{t.outlier} from Theorem \ref{t.twopaths}.

\begin{proof}[Proof of Theorem \ref{t.outlier}]
  For a $\pi$-stationary trajectory $X_0,\cdots,X_k$, let us define $p_{j,\ep}^k(\sigma)$ to be the probability that $\om(X_j)$ is in the bottom $\ep$ fraction of the values $\om(X_0),\dots,\om(X_k)$, \emph{conditioned} on the event that $X_j=\sigma$.

  In particular, to prove Theorem \ref{t.outlier}, we will prove the following claim:\\
  \textbf{Claim:} If $\sigma_0$ is not an $(\ep,\a)$-outlier, then
  \begin{equation}\label{l.claim}
  p_{0,\ep}^{k}(\sigma_0)\leq \sqrt{\frac{2\ep}{\alpha}}.
  \end{equation}
  Let us first see why the claim implies the theorem.  Recall the random variable $\rho$ is the number of trajectories $\cT^{i}$ from $\sigma_0$ on which $\sigma_0$ is observed to be an $\ep$-outlier with respect to the labeling $\omega$.  The random variable $\rho$ is thus a sum of $m$ independent Bernoulli random variables, which each take value 1 with probability $\leq \sqrt{\frac{2\ep}{\alpha}}$ by the claim.  In particular, by Chernoff's bound, we have
  \begin{equation}\label{l.chernoff}
    \Pr\brac{\rho\geq (1+\d) m\sqrt{\tfrac{2\ep}{\alpha}}}\leq e^{-\min(\d,\d^2) m\sqrt{\frac{2\ep}{\alpha}}/3},
  \end{equation}
 giving the theorem.  (Note the key point of the claim is that $\alpha$ is \emph{inside} the square root in \eqref{l.claim}, while a straightforward application of of Theorem \ref{t.rootep} would give an expression with $\alpha$ outside the square root.)
  
To prove \eqref{l.claim}, consider a $\pi$-stationary trajectory $X_0,\dots,X_k,\dots,X_{2k}$ and condition on the event that $X_k=\sigma$ for some arbitrary $\sigma\in \Sigma$.  Since $\cM$ is reversible, we can view this trajectory as two independent trajectories $X_{k+1},\dots,X_{2k}$ and $X_{k-1},X_{k-2},\dots,X_0$ both beginning from $\sigma$.  In particular, letting $A$ and $B$ be the events that $\om(X_k)$ is an $\ep$-outlier among the lists $\om(X_0),\dots,\om(X_k)$ and $\om(X_k),\dots,\om(X_{2k})$, respectively, we have that
  \begin{equation}
    p_{0,\ep}^k(\sigma)^2=\Pr(A\cap B)\leq p_{k,\ep}^{2k}(\sigma).
    \label{l.comparison}
  \end{equation}
  Now, the assumption that the given $\sigma_0\in \Sigma$ is not an $(\ep,\alpha)$-outlier gives that for a random $\sigma\sim \pi$, we have that
  \begin{equation}
    \label{l.notoutlier}
  \Pr\brac{p_{0,\ep}^k(\sigma)\geq p_{0,\ep}^k(\sigma_0)}\geq \alpha.
  \end{equation}

Line \ref{l.comparison} gives that $p_{0,\ep}^k(\sigma)^2\leq p_{k,\ep}^{2k}(\sigma)$, and Theorem \ref{t.twopaths} gives that $p_{k,\ep}^{2k}\leq 2\ep$.  Thus taking expectations with respect to a random $\sigma\sim \pi$, we obtain that 
  \[
  \E_{\sigma\sim \pi} \brac{p_{0,\ep}^k(\sigma)^2}\leq \E_{\sigma\sim \pi} \brac{p_{k,\ep}^{2k}(\sigma)}=p_{k,\ep}^{2k}\leq 2\ep.
  \]
  On the other hand, we can use \eqref{l.notoutlier} to write
  \[
  \E_{\sigma\sim \pi} \brac{p_{0,\ep}^k(\sigma)^2}\geq \alpha \cdot p_{0,\ep}^k(\sigma_0)^2,
  \]
  so that we have
  \[
p_{0,\ep}^k(\sigma_0)^2\leq \frac{2\ep}{\alpha}.
  \]
\end{proof}

The proof of Theorem \ref{t.GCoutlier} is quite similar:
\begin{proof}[Proof of Theorem \ref{t.GCoutlier}]
  For a $\pi$-stationary trajectory $X_0,\cdots,X_k$ and a real number $\mu$, let us define $p_{0,\ep}^\mu(\sigma)$ to be the probability that $\om(X_j)$ is in the bottom $\ep$ fraction of the values $\om(X_0),\dots,\om(X_k)$, \emph{conditioned} on the event that $X_0=\sigma$, where the length $k$ is chosen from a geometric distribution with mean $\mu$ supported on 0,1,2,\dots; i.e., $k=t$ with probability $\tfrac{1}{\mu+1}(1-\tfrac{1}{\mu+1})^t$.
  
  To prove Theorem \ref{t.GCoutlier}, it suffices to prove that if $\sigma_0$ is not an $(\ep,\a)$-outlier with respect to $k$ drawn from the geometric distribution with mean $\mu$, then
  \begin{equation}\label{l.GCclaim}
  p_{0,\ep}^\mu(\sigma_0)\leq \sqrt{\frac{\ep}{\alpha}}.
  \end{equation}

  To prove \eqref{l.GCclaim}, suppose that $k_1$ and $k_2$ are independent random variables which are are geometrically distributed with mean $\mu$, and consider a $\pi$-stationary trajectory
  \[
  X_0,\dots,X_{k_1},\dots,X_{k_1+k_2}
  \]
  of random length $k_1+k_2$, and condition on the event that $X_{k_1}=\sigma$ for some arbitrary $\sigma\in \Sigma$.  Since $\cM$ is reversible, we can view this trajectory as two independent trajectories $X_{k_1}, X_{k_1+1},\dots,X_{k_1+k_2}$ and $X_{k_1}, X_{k_1-1},X_{k_1-2},\dots,X_0$ both beginning from $X_{k_1}=\sigma$, of random lengths $k_2$ and $k_1$, respectively.  In particular, letting $A$ and $B$ be the events that $\om(X_{k_1})$ is an $\ep$-outlier among the lists $\om(X_0),\dots,\om(X_{k_1})$ and $\om(X_{k_1}),\dots,\om(X_{k_1+k_2})$, respectively, we have that
  \begin{multline}
    p_{0,\ep}^\mu(\sigma)^2=\Pr(A\cap B)
    \\\leq \Pr\big(\om(X_{k_1})\text{ is an $\ep$-outlier among }\om(X_0),\dots,\om(X_{k_1+k_2})\mid X_{k_1}=\sigma\big)
    \label{l.GCcomparison}
  \end{multline}
where, in this last expression, $k_1$ and $k_2$ are random variables.  Now, the assumption that the given $\sigma_0\in \Sigma$ is not an $(\ep,\alpha)$-outlier gives that for a random $\sigma\sim \pi$, we have that
  \begin{equation}
    \label{l.GCnotoutlier}
  \Pr\brac{p_{0,\ep}^\mu(\sigma)\geq p_{0,\ep}^\mu(\sigma_0)}\geq \alpha.
  \end{equation}
  Thus we write
  \begin{multline}\label{l.almostdone}
    \alpha \cdot p_{0,\ep}^\mu(\sigma_0)^2 \leq \E_{\sigma\sim \pi}\brac{p_{0,\ep}^\mu(\sigma)^2}\\\leq
    \Pr\big(\om(X_{k_1})\text{ is an $\ep$-outlier among }\om(X_0),\dots,\om(X_{k_1+k_2})\big),
  \end{multline}
  where the last inequality follows from line \eqref{l.GCcomparison}.

  On the other hand, considering the righthand side of Line \eqref{l.almostdone}, we have that conditioning on any value for the length $\ell=k_1+k_2$ of the trajectory, $k_1$ is uniformly distributed in the range $\{0,\dots,\ell\}.$  This is ensured by the geometric distribution, simply because for any $\ell$ and any $x\in \brac{0,\dots,\ell}$, we have that the probability
  \[
  \Pr\big (k_1=x\text{ AND } k_2=\ell-x\big )=\tfrac{1}{\mu+1}(1-\tfrac 1 {\mu+1})^{x}\tfrac{1}{\mu+1}(1-\tfrac 1 {\mu+1})^{\ell-x}=\brac{\tfrac {1}{\mu+1}}^2\brac{1-\tfrac 1 {\mu+1}}^\ell
  \]
  is independent of $x$.  In particular, conditioning on any particular value for the length $\ell=k_1+k_2$, we have that the probability that $\omega(X_{k_1})$ is an $\ep$-outlier on the trajectory is at most $\ep$, since $X_{k_1}$ is a uniformly randomly chosen element of the trajectory $X_0,\dots,X_{k_1+k_2}$; note that this part of the proof is exactly the same as the proof of Theorem \ref{t.GCserial}.  In particular, for the righthand-side of line \eqref{l.almostdone}, we are writing
  \begin{multline}
    \alpha\cdot p_{0,\ep}^\mu(\sigma_0)^2
    \leq \Pr\big(\om(X_{k_1})\text{ is an $\ep$-outlier among }\om(X_0),\dots,\om(X_{k_1+k_2})\big)\\\leq
    \max_{\ell} \Pr\big(\om(X_{k_1})\text{ is an $\ep$-outlier among }\om(X_0),\dots,\om(X_{k_1+k_2})\big\vert k_1+k_2=\ell\big)\\\leq \ep.
    \label{l.GCep}
  \end{multline}
  This gives line \eqref{l.GCclaim} and completes the proof.
  \end{proof}

We close this section by noting that in implementations where $m$ is not enormous, it may be sensible to use the exact binomial tail in place of the Chernoff bound in \eqref{l.chernoff}.  In particular, this gives the following versions:
\begin{theorem}\label{t.exact}    With $\rho$ as in Theorem \ref{t.outlier}, we have that if $\s_0$ is not an $(\ep,\alpha)$ outlier, then
  \begin{equation}
  \Pr\brac{\rho\geq K}\leq \sum_{k=K}^m \binom{m}{k}\bfrac{2\ep}{\alpha}^{k/2}\brac{1-\sqrt{\frac{2\ep}{\alpha}}}^{m-k}.
  \end{equation}
\end{theorem}
\begin{theorem}\label{t.randexact}    With $\rho$ as in Theorem \ref{t.GCoutlier}, we have that if $\s_0$ is not an $(\ep,\alpha)$ outlier, then
  \begin{equation}
  \Pr\brac{\rho\geq K}\leq \sum_{k=K}^m \binom{m}{k}\bfrac{\ep}{\alpha}^{k/2}\brac{1-\sqrt{\frac{\ep}{\alpha}}}^{m-k}.
  \end{equation}
\end{theorem}
\section{Generalizing the Besag and Clifford tests}
\label{s.starsplit}

Theorem \ref{t.outlier} is attractive because it succeeds at separating statistical significance from effect size, and at demonstrating statistical significance for an intuitively-interpretable deterministic property of state in the Markov Chain.  This is especially important when public-policy decisions must be made by non-experts on the basis of such tests.

In some cases, however, these may not be important goals.  In particular, one may simply desire a statistical test which is as effective as possible at disproving the null hypothesis $\sigma\sim \pi$.  This is a task at which Besag and Clifford's Theorem \ref{t.GCserial} excels.  

In their paper, Besag and Clifford also prove the following result, to enable a test designed to take efficient advantage of parallelism:
\begin{theorem}[Besag and Clifford parallel test]\label{t.GCparallel}
 Fix numbers $k$ and $m$.  Suppose that $\sigma_0$ is chosen from a stationary distribution $\pi$ of the reversible Markov Chain $\cM$, and suppose we sample a trajectory $X_1,X_2,\dots,X_k$ from $X_0=\sigma_0$, and then branch to sample $m-1$ trajectories $Z^s_1,Z^s_2,\dots,Z^s_k$ $(2\leq s\leq m)$ all from the state $Z^s_0=X_k$.   Then we have that
  \[
\Pr\brac{\om(\s_0)\text{\textnormal{ is an $\ep$-outlier among }}\om(\s_0),\om(Z^2_k),\om(Z^3_k),\dots,\om(Z^m_k)}\leq \ep.
  \]
\end{theorem}
\begin{proof}
  For this theorem it suffices to observe that $\sigma_0,Z^2_k,\dots,Z^m_k$ are \emph{exchangable} random variables---that is, all permutations of the sequence $\s_0,Z^2_k,\dots,Z^m_k$ are identical in distribution.  This is because if $\s_0$ is chosen from $\pi$ and then the $Z^i_k$'s are chosen as above, the result is equivalent in distribution to the case where $X_k$ is chosen from $\pi$ and then each $Z^i_k$ is chosen (independently) as the end of a trajectory $X_k,Z^i_{1},\dots,Z^i_k$, and $\s_0=Y_k$ is chosen (independently) as the end of a trajectory $X_k,Y_1,\dots,Y_k$.  Here we are using that reversibility implies that $(X_k,Y_1,\dots,Y_k)$ is identical in distribution to $(\s_0,X_1,\dots,X_k)$.
\end{proof}

With an eye towards finding a common generalization of Besag and Clifford's serial and parallel tests, we define a \emph{Markov outlier test} as a significance test with the following general features:
\begin{itemize}
\item The test begins from a state $\sigma_0$ of the Markov Chain which, under the null hypothesis, is assumed to be stationary;
\item random steps in the Markov chain are sampled from the initial state and/or from subsequent states exposed by the test;
\item the ranking of the initial state's label is compared among the labels of some (possibly all) of the visited states; it is an $\ep$-outlier if it's label is among the bottom $\ep$ of the comparison labels.  Some function $\rho(\ep)$ assigns valid statistical significance to the test results, as in the above theorems.
\end{itemize}
In particular, such a test may consist of single or multiple trajectories, may branch once or multiple times, etc.  In this section, we prove the validity of a parallelizable Markov outlier test with best possible function $\rho(\ep)=\ep$, but for which it is natural to expect the \emph{$\ep$-power} of the test---that is, its tendency to return small values of $\ep$ when $\sigma_0$ truly is an outlier---surpasses that of Theorems \ref{t.GCserial} and \ref{t.GCparallel}.  In particular, we prove the following theorem:
\begin{theorem}[Star-split test]\label{t.starsplit}
Fix numbers $m$ and $k$.  Suppose that $\s_0$ is chosen from a stationary distribution $\pi$ of the reversible Markov Chain $\cM$, and suppose that $\xi$ is chosen randomly in $\{1,\dots,k\}$.  Now sample trajectories $X_1,\dots,X_{\xi}$ and $Y_1,\dots,Y_{k-\xi}$ from $\s_0$, and then branch and sample $m-1$ trajectories $Z^s_1,Z^s_2,\dots,Z^s_k$ $(2\leq s\leq m)$ all from the state $Z^s_0=X_\xi$.   Then we have that
  \begin{align*}
    \Pr\Big(\om(\s_0)\text{\textnormal{ is an $\ep$-outlier among }}\om(\s_0),\,&\om(X_1)\dots,\om(X_{\xi-1}),\\
        &\om(Y_1),\dots,(Y_{k-\xi}),\\
        &\om(Z^2_1),\dots,\om(Z^2_k)\\
        &\vdots\\
        &\om(Z^m_k),\dots,\om(Z^m_k)\Big)\leq \ep.
  \end{align*}
\end{theorem}
In particular, note that the set of comparison random variables used consists of all random variables exposed by the test \emph{except} $X_{\xi}$.

To compare Theorem \ref{t.starsplit} with Theorems \ref{t.GCparallel} and \ref{t.GCserial}, let us note that it is natural to expect the $\ep$-power of a Markov chain significance test to depend on:
\begin{enumerate}[(a)]
\item How many comparisons are generated by the test, and
\item \label{howfar} how far typical comparison states are from the state being tested, where we measure distance to a comparison state by the number of Markov chain transitions which the test used to generate the comparison.
\end{enumerate}

If unlimited parallelism is available, then the Besag/Clifford parallel test is essentially optimal from these parameters, as it draws an unlimited number of samples, whose distance from the initial state is whatever serial running time is used.  Conversely, in a purely serial setting, the Besag/Clifford test is essentially optimal with respect to these parameters.

But it is natural to expect that even when parallelism is available, the number $n$ of samples we desire will often be be significantly greater than the parallelism factor $\ell$ available.  In this case, the Besag/Clifford parallel test will use $n$ comparisons at distance $d\approx \ell t/n$, where $t$ is the serial time used by the test.  In particular, the typical distance to a comparison can be considerably less than $t$ when $\ell$ compares unfavorably with $n$.

On the other hand, Besag/Clifford serial test generates comparisons whose typical distance is roughly $t/2$, but cannot make use of parallelism beyond $\ell=2$.  For an apples-to-apples comparison, it is natural to consider the case of carrying out their serial test using only every $d$th state encountered as a comparison state for some $d$.  This is equivalent to applying the test to the $d$th-power of the Markov chain, instead of applying it directly.  (In practical applications, this is a sensible choice when comparing the labels of states is expensive relative to the time required to carry out transitions of the chain.)  Now if $\ell$ is a small constant, we see that with $t\cdot d$ steps, the BC parallel test can generate roughly $n$ comparisons all at distance $d$ from the state being tested, the serial test could generate comparisons at distances $d,2d,3d,\dots,kd$ (measured in terms of transitions in $\cM$), where these distances occur with multiplicity at most 2, and $k=\max(\xi,n-\xi)\geq n/2$.  In particular, the serial test generates a similar number of comparisons in this way but at much greater distances from the state we are evaluating, making it more likely that we are able to detect that the input state is an outlier.

Consider now the star-split test.  Again, to facilitate comparison, we suppose the test is being applied to the $d$th power of $\cM$.  If serial time $t\approx sd$ is to be used, then we will branch into $\ell-1$ trajectories after $\xi\cdot$ $\cM^d$ chain, where $\xi$ is randomly chosen from $\{0,\tfrac s 2\}$.  Thus comparisons used lie at a set of distances $d,2d,\dots,(\xi+\tfrac s {2})d$ similar to the case of the Besag/Clifford serial test above.   But  now the distances $d,2d,\dots,(\xi d-1)d$ will have multiplicities at most 2 in the set of comparison distances, while the distances $(\xi+1)d, (\xi+2)d, \dots, (\xi+\tfrac s 2)d$ all have multiplicity at least $\ell-1$.  In particular, the test allows us to make more comparisons to more distance states, essentially by a factor of the parallelism factor being used.  In particular, it is natural to expect performance to improve as $\ell$ increases.  Moreover, the star-split test is equivalent to the Besag/Clifford serial test for $\ell\leq 2$, and essentially equivalent to their parallel test in the large $\ell$ limit.  (To make this latter correspondence exact, once can apply Theorem \ref{t.starsplit} to the $d$th power of a Markov chain $\cM$, and take $k=1$.)

\bigskip

We now turn to the task of proving Theorem \ref{t.starsplit}.  Unlike Theorems \ref{t.rootep}, \ref{t.twopaths}, and \ref{t.GCserial}, the comparison states used in Theorems \ref{t.GCparallel} and \ref{t.starsplit} cannot be viewed as a single trajectory in $\cM$. This motivates the natural generalization of the notion of a $\pi$-stationary trajectory as follows:
\begin{definition}\label{d.statT}   Given a reversible Markov Chain $\cM$ with stationary distribution $\pi$ and an undirected tree $T$, a \emph{$\pi$-stationary $T$-projection} is a collection of random variables $\{X_v\}_{v\in T}$ such that:
  \begin{enumerate}[(i)]
  \item \label{part.pi} for all $v\in T$, $X_v\sim \pi$;
  \item \label{part.edge} for any edge $\{u,v\}$ in $T$, if we let $T_u$ denote the vertex-set of the connected component of $u$ in $T\setminus \{u,v\}$ and $\{\s_w\}_{w\in T}$ is an arbitrary collection of states, then
\[
\Pr\brac{X_v=\s_v\middle| \bigwedge_{w\in T_u} X_w=\s_w}=\pi_{\s_u,\s_v}.
\]
  \end{enumerate}
\end{definition}


In analogy to the case of $\pi$-stationary trajectories, Definition \ref{d.statT} easily gives the following, by induction:
\begin{observation}\label{o.Tunique}
  For fixed $\pi$ and $T$, if $\{X_w\}_{w\in T}$ and $\{Y_w\}_{w\in T}$ are both $\pi$-stationary $T$-projections, then the two collections $\{X_w\}_{w\in T}$ and $\{Y_w\}_{w\in T}$ are equivalent in distribution.\qed
\end{observation}
This enables the following natural analog of Definition \ref{d.rhos}:
\begin{definition}
  Given a Markov Chain $\cM$ with labels $\omega:\Sigma\to \Re$ and stationary distribution $\pi$, we define for each $\ell$, each undirected tree $T$, each vertex subset $S\subset T$ and each vertex $v\in S$ a real number $\r_{v,\ell}^{T,S}$, which is the probability that for a $\pi$-stationary $T$-projection $\{X_w\}_{w\in T}$, we have that  $\omega(X_v)$ is $\ell$-small among $\{\om(X_w)\}_{w\in S}$.
\end{definition}

Observe that as in \eqref{e.sum} we have for any tree $T$ and any vertex subset $S$ of $T$, we have that
\begin{equation}\label{e.Tsum}
  \sum_{w\in S}\rho^{T,S}_{w,\ell}\leq \ell+1.
\end{equation}

The following Observation, applied recursively, gives the natural analog of Observation \ref{o.stat}.  Again the proof is an easy exercise in the definitions.
\begin{observation}\label{o.rec}
  Suppose that $T$ is an undirected tree, $v$ is a leaf of $T$, $T'=T\setminus v$, and $\{X_w\}_{w\in T'}$ is a $\pi$-stationary $T'$-projection.  Suppose further that $X_v$ is a random variable such that for all $\{\s_w\}_{w\in T}$ we have that
\begin{equation}\label{e.uv}
\Pr\brac{X_v=\s_v\middle| \bigwedge_{w\in T'} \brac{X_w=\s_w}}=\pi_{\s_u,\s_v},
\end{equation}
where $u$ is the neighbor of $v$ in $T$.   Then $\{X_w\}_{w\in T}$ is a $\pi$-stationary $T$-projection.\qed
\end{observation}

We can rephrase the proof of Theorem \ref{t.GCparallel} in this language.  Let $T$ be the tree consisting of $m$ paths of length $k$ sharing a common endpoint and no other vertices, and let $S$ be the leaves of $T$.  By symmetry, we have that $\r^{T,S}_{w,\ell}$ is constant over $w\in S$.  On the other hand, Observation \ref{o.rec} gives that under the hypotheses of Theorem \ref{t.GCparallel}, $\s_0$, $X_1,\dots,X_k$, and the $Z^s_i$'s are a $\pi$-stationary $T$-projection, with obvious assignments (e.g., $\s_0$ corresponds to a leaf of $T$; $X_k$ corresponds to the center).  In particular, \eqref{e.Tsum} implies that $\r^{T,S}_{w,\ell}\leq \frac{\ell+1}{n}$, which gives the theorem.

On the other hand, the definitions makes the following proof easy as well, using the same simple idea as Besag and Clifford's Theorem \ref{t.GCserial}.
\begin{proof}[Proof of Theorem \ref{t.starsplit}]
  Define $T$ to be the undirected tree with vertex set $\{v_0\}\cup \{v^s_j\mid 1\leq s\leq m, 1\leq j\leq k\}$, with edges $\{v_0,v^s_1\}$ for each $1\leq s\leq m$ and $\{v^s_j,v^s_{j+1}\}$ for each $1\leq s\leq m$, $1\leq j\leq k-1$.
 Now we let $S$ consist of all vertices of $T$ except the center $v_0$, and let $S_j$ denote the set of $m$ vertices in $S$ at distance $j$ from $v_0$.  By symmetry, we have that $\rho^{T,S}_{v,\ell}$ is constant in each $S_j$; in particular, we have that
  \[
  \rho^{T,S}_{v^1_j,\ell}=\frac{1}{n} \sum_{s=1}^m \rho^{T,S}_{v^s_j,\ell}
  \]
  and together with \eqref{e.Tsum} this gives that
  \begin{equation}\label{e.treeleq}
  \sum_{j=1}^k\rho^{T,S}_{v^1_j,\ell}\leq \frac{\ell+1}{n}.
  \end{equation}

  Now if we let
  \begin{align*}
    W_{v_0}&=X_k,\\
    W_{v^s_j}&=\begin{cases}
       X_{\xi-j} & s=1, 1\leq j<\xi\\
       \s_0 & s=1, j=\xi\\
       Y_{j-\xi} & s=1, j>\xi\\
       Z^s_j & 2\leq s\leq m, 1\leq j\leq k,
    \end{cases}
  \end{align*}
  then $\{W_w\}_{w\in T}$ is a $\pi$-stationary $T$-projection under the hypotheses of Theorem \ref{t.starsplit}, by recursively applying Observation \ref{o.rec}.  Moreover, as $\xi$ is chosen randomly among $\{1,\dots,k\}$, the probability that $\om(\s_0)=\om(W_{v^1_\xi})$ is $\ell$-small among $\{\om(W_w)\}_{w\in S}$ is given by
  \[
  \frac 1 k\brac{\rho^{T,S}_{v^1_1,\ell}+\dots+\rho^{T,S}_{v^1_k,\ell}}\leq \frac{\ell+1}{kn},
  \]
  where the inequality is from \eqref{e.treeleq}, giving the Theorem.
\end{proof}

\section{The product space setting}
\label{s.product}

The appeal of the theorems developed thus far in this paper is that they can be applied to any reversible Markov chain without any knowledge of its structure.  However, there are some important cases where additional information about the structure of the stationary distribution of a chain \emph{is} available, and can be exploited to enable more powerful statistical claims.

In this section, we consider the problem of evaluating claims of gerrymandering with a Markov Chain where the probability distribution on districtings is known to have a product structure imposed by geographical constraints.  For example, the North Carolina Supreme Court has ruled in \emph{Stephenson v.~Bartlett} that districtings of that state must respect groupings of counties determined by a prescribed algorithm.  In particular a set of explicit rules (nearly) determine a partition of the counties of North Carolina into county groupings whose populations are each close to an integer multiple of an ideal district size  (see \cite{countyclusters} for recent results on these rules), and then the districting of the state is comprised of independent districtings of each of the county groupings.

In this way, the probability space of uniformly random districtings is a product space, with a random districting of the whole state equivalent to collection of random independent districtings of each of the separate county groupings.  We wish to exploit this structure for greater statistical power.  In particular, running trajectories of length $k$ in each of $d$ clusters generates a total of $k^d$ comparison maps with only $k\cdot d$ total Markov chain steps.  To take advantage of the potential power of this enormous comparison set, we need theorems which allow us to compare a given map not just to a trajectory of maps in a Markov chain (since the $k^d$ maps do not form a trajectory) but to the product of trajectories.  This is what we show in this section.

Formally, in the product space setting, we have a collection $\cM^{[d]}$ of $d$ Markov Chains $\cM_1,\dots,\cM_d$, each $\cM_i$ on state space $\Sigma_i$ (each corresponding to one county grouping in North Carolina, for example).  We are given a label function $\omega:\Sigma^{[d]}\to \Re$, where here $\Sigma^{[d]}=\Sigma_1\times\dots\times\Sigma_d$.  In the first theorem in this section, which is a direct analog of the Besag and Clifford test, we consider a $\bm \sigma_0\in \Sigma^{[d]}$ distributed as $\bm \sigma_0\sim \pi^{[d]}$, where here $\pi^{[d]}$ indicates the product space of stationary distributions $\pi_i$ of the $\cM_i$.  (In the gerrymandering case, $\pi^{[d]}$ is a random map chosen by randomly selecting a map for each separate county cluster.)  In the tests discussed earlier in this paper, a state $\sigma_0\sim \cM$ is evaluated by comparing a state $\sigma_0$ to other states on a trajectory containing $\sigma_0$.  In the product setting, we compare $\bm \sigma_0$ against a product of one trajectory from each $\cM_i$.

In particular, given the collection $\cM^{[d]}$, a state $\bm\sigma_0=(\sigma_0^1,\dots,\sigma_0^d)\in \Sigma^{[d]}$, and $\mathbf j=(j_1,\dots j_d)$, $\mathbf k=(k_1,\dots,k_d)$, we define the \emph{trajectory product} $\mathbf X_{\bm \sigma_0,\mathbf j,\mathbf k}$ which is obtained by considering, for each $i$, a trajectory $X_0^i,\dots,X_{k^i}^i$ in $\cM_i$ conditioned on $X^i_{j_i}=\sigma_0^i$.   $\mathbf X_{\bm \sigma_0,\mathbf j,\mathbf k}$ is simply the set of all $d$-tuples consisting of one element from each such trajectory.

We define the \emph{stationary trajectory product} $\mathbf X_{\pi^{[d]},\mathbf k}$, analogously, except that the trajectories used are all stationary, instead of conditioning on $X^i_{j_i}=\sigma_0^i$.

\begin{theorem}\label{t.GCproduct}
  Given reversible Markov Chains $\cM_1,\cM_2,\dots,\cM_d$, fix any number $k$ and suppose that $\sigma^1_0,\dots,\sigma^d_0$ are chosen from stationary distributions $\pi_1,\dots,\pi_d$ of $\cM_1,\dots,\cM_d$, and that $\xi_1,\dots,\xi_d$ are chosen uniformly and independently in $\{0,\dots,k\}$.  For each $s=1,\dots,d$, consider two independent trajectories $Y^s_0,Y^s_1,\dots$ and $Z^s_0,Z^s_1,\dots$ in the reversible Markov Chain $\cM_s$ from $Y^s_0=Z^s_0=\sigma^s_0$. Let $\omega:\cM_1\times\dots\times\cM_d\to \Re$ be a label function on the product space, write $\bm\sigma_0=(\sigma_0^1,\dots,\sigma_0^d)$, and denote by $\mathbf Z_{\bm \sigma_0,k}$ the (random) set of all vectors $(a_1,\dots,a_d)$ such that for each $i$, $a_i\in \brac{\sigma_0^i,Y^i_1,\dots,Y^i_{\xi_i},Z^i_1,\dots,Z^i_{k-\xi_i}}$.  Then we have that 
  \begin{equation}\label{e.GCproduct}
\Pr\brac{\om(\bm \s_0)\text{\textnormal{ is an $\ep$-outlier among }}\om(\mathbf x), \mathbf x \in \mathbf Z_{\mathbf \sigma_0,k}}\leq\ep.
  \end{equation}
\end{theorem}
\begin{proof}Like the proof of Theorem \ref{t.GCserial}, this proof is very simple; it is just a matter of digesting notation.  First observe that $\mathbf Z_{\bm \sigma_0,k}$ is simply a trajectory product $\mathbf X_{\bm \sigma_0,\mathbf \xi,\mathbf k}$, where where $\mathbf k=(k,\dots,k)$ and $\mathbf{\xi}$ is the random variable $(\xi_1,\dots,\xi_d)$.
  
  In particular, under the hypothesis that $\sigma^i_0\sim \pi_i$ for all $i$, $\mathbf Z_{\bm \sigma_0,k}$ is in fact a stationary trajectory product $\mathbf X_{\pi^{[d]},  \mathbf k}$,   In particular, by the random, independent choice of the $\xi_i$'s, the probability in \eqref{e.GCproduct} is equivalent to the probability that the label of a random element of the a stationary trajectory product is among $\ep$ smallest labels in the stationary trajectory product; this probability is at most $\ep$.
\end{proof}
The following is an analog of Theorem \ref{t.twopaths} for the product space setting.

\begin{theorem}\label{t.twopathproduct}
  Given reversible Markov Chains $\cM_1,\cM_2,\dots,\cM_d$, fix any
  number $k$ and suppose that $\sigma^1_0,\dots,\sigma^d_0$ are chosen
  from stationary distributions $\pi_1,\dots,\pi_d$ of
  $\cM^1,\dots,\cM^d$.  For each $s=1,\dots,d$, consider two
  independent trajectories $Y^s_0,Y^s_1,\dots$ and $Z^s_0,Z^s_1,\dots$
  in the reversible Markov Chain $\cM^s$ from
  $Y^s_0=Z^s_0=\sigma^s_0$. Let
  $\omega:\cM_1\times\dots\times\cM_d\to \Re$ be a label function on
  the product space, write
  $\bm\sigma_0=(\sigma_0^1,\dots,\sigma_0^d)$, and denote by
  $\mathbf Z_{\bm \sigma_0,k}$ the (random) set of all vectors
  $(a_1,\dots,a_d)$ such that for each $i$,
  $a_i\in \brac{\sigma_0^i,Y^i_1,\dots,Y^i_{k},Z^i_1,\dots,Z^i_{k}}$.
  Then we have that
  \begin{equation}\label{e.twopathproduct}
\Pr\brac{\om(\bm \s_0)\text{\textnormal{ is an $\ep$-outlier among }}\om(\mathbf x), \mathbf x \in \mathbf Z_{\bm \sigma_0,k}}\leq2^d\cdot \ep.
  \end{equation}
\end{theorem}

\begin{proof}
  First consider $d$ independent stationary trajectories $X_0^i,X_1^i,X_2^i,\dots$ for each $i=1,\dots,d$, and define $\mathbf X_{\pi,k}$ to be the collection of all $(k+1)^d$ $d$-tuples $(a_1,\dots,a_d)$ where, for each $i$, $a_i\in \{X^i_0,\dots,X^i_k\}$.
  
  In analogy to Definition \ref{d.rhos}, we define $\rho^k_{\mathbf{j},\ell}$ for $\mathbf{j}=(j_1,j_2,\dots,j_k)$ to be the probability that for $X_{\mathbf j}=(X^1_{j_1},\dots,X^d_{j_d})\in \mathbf X_{\pi,k}$, we have that $\omega(X_{\mathbf j})$ is $\ell$-small among the $\omega$-labels of all elements of $\mathbf X_{\pi,k}$.

  Observe that for $\mathbf k=(k,\dots,k)$, we have in analogy to equation \eqref{l.compare} that
  \begin{equation}\label{l.dcompare}
    \rho^{2k}_{\mathbf k,\ell}\leq \rho^k_{\mathbf j,\ell}
  \end{equation}
  for any $j=(j_1,\dots,j_d)$.  And of course we have that
  \[
  \sum_{\mathbf j}\rho_{\mathbf j,\ell}^k\leq \ell+1.
  \]
  Thus averaging both sides of \eqref{l.dcompare} gives that
  \begin{equation}
  \label{l.ddone}
  \rho^{2k}_{\mathbf k,\ell}\leq \frac{\ell+1}{(k+1)^d}\leq 2^d\frac{\ell+1}{(2k+1)^d}.
  \end{equation}
  Now observe that the the statement that
  \[
  \omega(\bm \sigma_0)\text{ is an $\ep$-outlier among }\omega(\mathbf x),\mathbf x\in \mathbf Z_{\bm \sigma_0,k}
  \]
  equivalent to the statement that
    \[
  \omega(\bm \sigma_0)\text{ is an $\ell$-small among }\omega(\mathbf x),\mathbf x\in \mathbf Z_{\bm \sigma_0,k}
  \]
  for $\ell=\ep\cdot (2k+1)^d-1$; thus \eqref{l.ddone} gives the theorem, since $\rho^{2k}_{\mathbf k,\ell}$ is precisely the probability that this second statement holds.
\end{proof}
The presence of the $2^d$ in \eqref{e.twopathproduct} is now potentially more annoying than the constant 2 in \eqref{t.twopaths}, and it is natural to ask whether it can be avoided.  However, using the example from Remark \ref{r.badcase}, it is easy to see that an exponential factor $\bfrac 3 2 ^d$ may really be necessary, at least if $k=1$.  Whether such a factor can be avoided for larger values of $k$ is an interesting question.  However, as we discuss below, this seemingly large exponential penalty is actually likely dwarfed by the quantitative benefits of the product setting, in many real-world cases.

\subsection{Illustrative product examples}
\label{sec:illustrative-example}

The fact the estimate in Theorem~\ref{t.GCproduct} looks like original
Theorem~\ref{t.GCserial}, hides the power in the product version. More
misleading is the fact that Theorem~\ref{t.twopathproduct} has a
$2^d$ which seems to make the theorem degrade with increasing $d$.

Let us begin by considering the simplest example we are looking for
the single extreme outlier across the entire product space. Let us
further assume that this global extreme is obtained by choosing each
of the extreme element in each part of the product space. An example
of this comes for the Gerrymandering application where one is
naturally interested in the seat count. Each of the product
coordinates represents the seats from a particular geographic
region. In some states such as North Carolina judicial rulings break
the problem up into the product measure required by
Theorem~\ref{t.GCproduct}  and  Theorem~\ref{t.twopathproduct} by
stipulating that particular geographic regions must be redistricted
independently.

For illustrative purposes, lets assume that there are $L$ different outcomes
in each of the $d$ different factors of the product space.  Hence
the chance of getting the minimum in any of the $d$ different
components is $1/L$. However, getting the minimum in the whole
product space requires getting the minimum in each of the components and so is
$1/L^d$. Hence is this setting one can
take $\epsilon=  1/L^d$ in Theorem~\ref{t.GCproduct}  and
Theorem~\ref{t.twopathproduct}. Thus even in Theorem~\ref{t.twopathproduct} as long as $L > 2$, one has a
significant improvement as $d$ grows.

Now lets consider a second slightly more complicated example which
builds on the proceeding one. Let us equip each $\mathcal{M}_i$ with
a function $\omega_i$ and decide that we are interested in the event
\begin{equation}
  \label{eq:1}
  \mathcal{E}(\delta)=\Big\{ \{\xi_i\}_{1}^d \colon  \sum_{i=1}^d \omega_i(\xi_i) \leq \delta \Big\}\,.
\end{equation}
Then one can take
\begin{align*}
  \epsilon = \frac{|\mathcal{E}(\delta)|}{L^d}
\end{align*}
in Theorem~\ref{t.twopathproduct} and $2^d$ times this in
Theorem~\ref{t.twopathproduct}, where $|\mathcal{E}(\delta)|$ is
simply the number of elements in the set $\mathcal{E}(\delta)$. This can lead to a significant
improvement in the power of the test in the product case over the
general case when $|\mathcal{E}(\delta)|$ grows slower than $L^d$.

There remains the task of calculating $|\mathcal{E}(\delta)|$. In the gerrymandering examples
we have in mind, this can be done efficiently. When counting seat
counts, the map $\omega_i$ is a many--to--one map with a range consisting
of a few discrete values. This means that one can tabulate exactly the
number of samples which produce a given value of $\omega_i$. Since we
are typically interested extreme values of
\begin{align*}
  \omega(\xi) = \sum_{i=1}^d \omega_i(\xi_i)\,,
\end{align*}
there are often only a few partitions of  each value of $\omega$ made
from possible values of $\omega_i$. When this true, the size of
$\mathcal{E}$ can be calculated exactly efficiently.

For example, let us assume there are $d$ geographical regions which
each needs to be divided into 4 districts. Furthermore each party
always wins at least one seat in each geographical region; hence, the
only possible outcomes are 1, 2 or 3 seats in each region for a given
party. If  $\omega_i$ counts the number of seats for the party of
interest in geographic region $i$, let us suppose for concreteness
that we want are interested in $\delta = 2d$. To calculate
$|\mathcal{E}(\delta)|$, we need to only keep track of the number of
times 1, 2 or 3 seats is produced in each geographic region. We can
then combine these numbers by summing over all of the ways the numbers
1, 2 and 3 can add numbers between $d$ and  $2d$. (The smallest
$\omega(\xi)$ can be given our assumptions is $d$.) This is a straightforward calculation
for which there exist fast algorithms which leverage the hierarchical
structure. Namely, group each region with another and calculate
the combined possible seat counts and their frequencies.  Continuing up the tree recursively one can calculate
$|\mathcal{E}(\delta)|$ in only logarithmically many
levels. 

It is worth remarking, that not all statistics of interest fall as
neatly into this framework which enables simple and efficient computation. For instance, calculating
the ranked marginals used in \cite{MH} requires choosing some
representation of the histogram, such as a fixed binning, and would
yield only approximate results.

\subsection{Towards an $(\ep,\alpha)$-outlier theorem for product spaces}
\label{sec:more-disc-prod}

In general, the cost of making a straightforward translation of Theorems \ref{t.outlier} or \ref{t.GCoutlier} to the product-space setting are surprisingly large: in both cases, the square root is replaced by a $2^d$th root, according to the natural generalization of the proofs of those theorems.

Accordingly, in this section we point out simply that by using a more complicated definition of $(\ep,\alpha)$-outliers for the product space setting, an analog of Theorem \ref{t.GCoutlier} is then easy.  In particular, let us define 
\begin{equation}\label{pbD}
p^{\mathbf k}_{\mathbf{U},\ep}(\bm \sigma_0):=\Pr\big(\omega(\bm \sigma_0)\text{ an $\ep$-outlier in }\mathbf X_{\bm \sigma_0,\mathbf j,\mathbf k}\big),
\end{equation}
where $\mathbf{j}=(j_1,\dots,j_d)$ is chosen randomly with respect to the uniform distributions $j_i\sim \mathrm{Unif}[0,k_i]$ (here $\mathbf k=(k_1,\dots,k_d)$).

Now we define a state $\bm \sigma_0$ to be an $(\ep,\alpha)$-outlier with respect to a distribution $\mathbf k$ if among all states in $\Sigma^{[d]}$, we have that $p^{\mathbf k}_{\mathbf{U},\ep}(\bm \sigma_0)$ is in the the largest $\alpha$ fraction of the values of $p^{\mathbf k}_{\mathbf U,\ep}(\bm \sigma)$ over \emph{all} states $\bm \sigma\in \cM^{[d]}$, weighted according to $\pi$.

\begin{theorem}\label{t.productClaim}
  We are given Markov Chains $\cM_1,\dots,\cM_d$.   Suppose that $\mathbf \sigma_0$ is not an $(\ep,\alpha)$-outlier with respect to $\mathbf k$.  Then
  \[
  p^{\mathbf k}_{\mathbf U,\ep}(\sigma_0)\leq \frac\ep \alpha.
  \]
\end{theorem}
\begin{proof}
  This follows immediately from the definitions.  From the definition of $(\ep,\alpha)$-outlier given above for the product setting, we have that if $\sigma_0$ is not an $(\ep,\alpha)$-outlier, then for a random $\sigma\sim \pi$,
  \[
  \Pr\bigg(p^{\mathbf k}_{\mathbf U,\ep}(\sigma)\geq p^{\mathbf k}_{\mathbf U,\ep}(\sigma_0)\bigg)\geq \alpha.
  \]
  Thus we can write
  \[
  \E_{\sigma\sim \pi}p^{\mathbf k}_{\mathbf U,\ep}(\sigma)\geq \alpha \cdot p^{\mathbf k}_{\mathbf U,\ep}(\sigma_0).
  \]
  And of course this expectation is just the probability that a random element of $\mathbf X_{\pi,\mathbf k}$ is an $\ep$-outlier on $X_{\pi,\mathbf k}$, which is at most $\ep$.
\end{proof}
Of course this kind of trivial proof would be possible in the general non-product space setting also, but the sacrifice is that $(\ep,\alpha)$-outliers cannot be defined with respect to the endpoints of trajectories, which appears most natural.  Whether theorems analogous to \ref{t.outlier} and \ref{t.GCoutlier} are possible in the product space setting without an explosive dependence on the dimension $d$ seems like a very interesting question.

\end{document}